\documentclass[a4paper,12pt]{article}
\usepackage{amssymb,amsmath,amsthm,amsfonts,amsbsy,latexsym}
\usepackage{graphicx}
\usepackage[pdftex,bookmarks,colorlinks=false]{hyperref}
\usepackage[font=small,labelfont=bf]{caption}
\usepackage{subcaption}
\usepackage[hmargin=1.2in,vmargin=1.2in]{geometry}
\usepackage{mathrsfs}

\usepackage[mathscr]{euscript}

\usepackage[auth-lg]{authblk}
\def\ni{\noindent}
\newcommand{\N}{\mathbb{N}}
\newcommand{\J}{\mathscr{J}}

\newtheorem{theorem}{Theorem}[section]
\newtheorem{proposition}[theorem]{Proposition}
\newtheorem{corollary}[theorem] {Corollary}
\newtheorem{definition}[theorem]{Definition}
\newtheorem*{example}{Example}

\newtheorem{lemma}{Lemma}[section]

\title{\bf Johan Colouring of Graph Operations}

\author{\bf Johan Kok$^{\ast}$, Sudev Naduvath$^\dagger$}
\affil{\small Centre for Studies in Discrete Mathematics\\ Vidya Academy of Science and Technology\\ Thalakkottukara, Thrissur - 680501, India.\\ {\tt $^{\ast}$kokkiek2@tshwane.gov.za,\\ \tt $^\dagger$sudevnk@gmail.com}}

\date{}

\begin{document}
\maketitle

\begin{abstract}
A vertex $v$ of a given graph is said to be in a rainbow neighbourhood of $G$ if every colour class of $G$ consists of at least one vertex from the closed neighbourhood $N[v]$. A maximal proper colouring of a graph $G$ is a Johan colouring if and only if every vertex of $G$ belongs to a rainbow neighbourhood of $G$. In general all graphs need not have a Johan colouring, even though they admit a chromatic colouring. In this paper, we characterise graphs which admit a Johan colouring. We also discuss some preliminary results in respect of certain graph operations which admit a Johan colouring under certain conditions.
\end{abstract}

\ni \textbf{Keywords:} Johan colouring, modified Johan colouring, rainbow neighbourhood.

\vspace{0.25cm}

\ni \textbf{Mathematics Subject Classification:} 05C15, 05C38, 05C75, 05C85.

\section{Introduction}
For general notations and concepts in graphs and digraphs we refer to \cite{BM1,FH,DBW}. For further definitions in the theory of graph colouring, see \cite{CZ1,JT1}. Unless specified otherwise, all graphs mentioned in this paper are simple, connected  and undirected graphs. 

The degree of a vertex $v \in V(G)$ is the number of edges in $G$ incident with $v$ and is denoted $d_G(v)$ or when the context is clear, simply as $d(v)$. Also, unless mentioned otherwise, the graphs we consider in this paper has the order $n$ and size $p$ with minimum and maximum degree $\delta$ and $\Delta$, respectively. 

Recall that if $\mathcal{C}= \{c_1,c_2,c_3,\dots,c_\ell\}$ and $\ell$ sufficiently large, is a set of distinct colours, a proper \textit{vertex colouring} of a graph $G$ is a vertex colouring $\varphi:V(G) \mapsto \mathcal{C}$ such that no two distinct adjacent vertices have the same colour. The cardinality of a minimum set of colours which allows a proper vertex colouring of $G$ is called the \textit{chromatic number} of $G$ and is denoted by $\chi(G)$. When a vertex colouring is considered with colours of minimum subscripts, the colouring is called a \emph{minimum parameter} colouring. Unless stated otherwise, all colourings in this paper are minimum parameter colour sets. 

The number of times a colour $c_i$ is allocated to vertices of a graph $G$ is denoted by $\theta(c_i)$ and $\varphi:v_i \mapsto c_j$ is abbreviated, $c(v_i) = c_j$. Furthermore, if $c(v_i) = c_j$ then $\iota(v_i) = j$. The colour class of a colour $c_i$, denoted by ${\cal C}_i$, is the set of vertices of $G$ having the same colour $c_i$.

We shall also colour a graph in accordance with the Rainbow Neighborhood Convention (see \cite{KSJ}), which is stated as follows.

\textbf{Rainbow Neighbourhood Convention:} (\cite{KSJ}) For a proper colouring $\mathcal{C} =\{c_1,c_2,c_3,\dots,c_\ell\}$, $\chi(G) = \ell$, we always colour maximum possible number of vertices with the colour $c_1$, then colour the maximum possible number of remaining vertices by the colour $c_2$  and proceeding like this and finally colour the remaining vertices by the colour $c_\ell$. Such a colouring is called a $\chi^-$-colouring of a graph.

The inverse to the convention requires the mapping $c_j\mapsto c_{\ell -(j-1)}$. Corresponding to the inverse colouring we define $\iota'(v_i) = \ell-(j-1)$ if $c(v_i) = c_j$.  The inverse of a $\chi^-$-colouring is called a $\chi^+$-colouring.

The closed neighbourhood $N[v]$ of a vertex $v \in V(G)$ which contains at least one coloured vertex of each colour in the chromatic colouring, is called a \textit{rainbow neighbourhood}. That is, a vertex $V$ is said to be in a rainbow neighbourhood if ${\cal C}_i\cap N[v]\ne \emptyset$, for all $1\le i\le \chi(G)$. Two types of vertex colourings in terms of rainbow neighbourhoods have been introduced in \cite{NKS} as follows.

\begin{definition}\label{Def-1.1}{\rm 
\cite{NKS} A maximal proper colouring of a graph $G$ is a \textit{Johan colouring} of $G$, denoted by $J$-colouring, if and only if every vertex of $G$ belongs to a rainbow neighbourhood of $G$. The maximum number of colours in a $J$-colouring is called the \textit{$J$-chromatic number} of $G$, denoted by $\J(G)$.}
\end{definition}

\begin{definition}\label{Def-1.21}{\rm 
\cite{NKS} A maximal proper colouring of a graph $G$ is a \textit{modified Johan colouring}, denoted by $J^*$-colouring, if and only if every internal vertex of $G$ belongs to a rainbow neighbourhood of $G$. The maximum number of colours in a $J^*$-colouring is denoted by $\J^*(G)$.}
\end{definition}

In this paper, we characterise the graphs which admit Johan colouring. We also discuss some preliminary results in respect of certain graph operations which admit a Johan colouring under certain conditions. 

\section{New Directions}

A \textit{null graph} on $n$ vertices is an edgeless graph and is denoted by $\mathfrak{N}_n$. We follow the convention that $\J(\mathfrak{N}_n)=\J^*(\mathfrak{N}_n)=1,\ n \in \N$. Also, note that for any graph $G$ which admits a $J$-colouring, we have $\chi(G) \leq \J(G)$. 

Note that if a graph $G$ admits a $J$-colouring, it also admits a $J^*$-colouring. However, the converse need not be true always. It can also be noted that if graph $G$ has no pendant vertex and it admits a $J$-colouring, then $\J(G)=\J^*(G)$.

In view of the above mentioned concepts and facts, we have the following theorem.

\begin{theorem}\label{Thm-2.1}
If $G$ is a tree of order $n\geq 2$, then $\J(G)<\J^*(G)$.
\end{theorem}
\begin{proof}
A tree $G$ of order $n \geq 2$ has at least two pendant vertices, say $u$ and $v$. Therefore, the maximum number of colours which will allow both vertices $u$ and $v$ to yield rainbow neighbourhoods is $\chi(G)=2$. Therefore, $G$ admits a $J$-colouring and $\J(G)=2$.

Any internal vertex $w$ of $G$ has $d(w)\geq 2$. Therefore, $\J^*(G) \leq 3$. Consider any diameter path of $G$ say $P_{diam(G)}$. Beginning at a pendant vertex of the diameter path, label the vertices consecutively $v_1,v_2,v_3,\ldots,v_{diam(G)}$. Colour the vertices consecutively $c(v_1)=c_1$, $c(v_2)=c_2$, $c(v_3)=c_3$, $c(v_4) = c_1$, $c(v_5) = c_2$, $c(v_6) = c_3$ and so on such that

\begin{eqnarray}
c(v_{diam(G)})=1; & \text{if}\ diam(G) \equiv 1({\rm mod}~3)\\
c(v_{diam(G)})=2; & \text{if}\ diam(G) \equiv 2({\rm mod}~3)\\
c(v_{diam(G)})=3; & \text{if}\ diam(G) \equiv 0({\rm mod}~3).
\end{eqnarray}

Clearly, in respect of path $P_{diam(G)}$, it is a proper colouring and all internal vertices yield a rainbow neighbourhood on $3$ colours. Consider any maximal path starting from, say $v \in V(P_{diam(G)})$. Hence, $v$ is a pendant vertex to that maximal path. Colour the vertices consecutively from $v$ as follows:

\begin{enumerate}
\item[(a)] If $c(v) = c_1$ in $P_{diam(G)}$, colour as $c_1,c_2,c_3,c_1,c_2,c_3,\ldots,\underbrace{c_1~\text{or}~c_2~\text{or}~c_3}$
\item[(b)] If $c(v) = c_2$ in $P_{diam(G)}$, colour as $c_2,c_3,c_1,c_2,c_3,c_1,\cdots,\underbrace{c_2~\text{or}~c_3~\text{or}~c_1}$.
\item[(c)] If $c(v) = c_3$ in $P_{diam(G)}$, colour as $c_3,c_1,c_2,c_3,c_1,c_2,\cdots,\underbrace{c_3~\text{or}~c_1~\text{or}~c_2}$.
\end{enumerate}

It follows from mathematical induction that all maximal branching can receive such colouring which remains a proper colouring with all internal vertices $v\in V(G)$ having $|c(N[v])|=3$. Furthermore, all nested branching can be coloured in a similar way until all vertices of $G$ are coloured. Therefore, $\J^*(G) \geq 3$. Hence, $\J(G) < \J^*(G)$. 
\end{proof}

An easy example to illustrate Theorem \ref{Thm-2.1} is the star $K_{1,n},\ n \geq 2$ for which $\J(K_{1,n}) =2 < n+1 = \J^*(K_{1,n})$. This example prompts the next results.

\begin{corollary}
For any graph $G$ which admits a $J^*$-colouring, we have $\J^*(G) \leq \Delta(G)+1$.
\end{corollary}

\begin{corollary}
If $\J^*(G) > \J(G)$ for a graph $G$, then $G$ has at least one pendant vertex.
\end{corollary}
\begin{proof}
Since all $v\in V(G)$ are internal vertices and any vertex $u$ for which $d(u)=\delta(G)$ must yield a rainbow neighbourhood, it follows that any maximal proper colouring $\mathcal{C}$ are bound to $|\mathcal{C}|=|N[u]|=\delta(G)+1$. Therefore, if $\J^*(G)>\J(G)$, then $G$ has at least one pendant vertex.
\end{proof}

In \cite{KSJ}, the rainbow neighbourhood number $r_\chi(G)$ is defined as the number of vertices of $G$ which yield rainbow neighbourhoods. It is evident that not all graphs admit a $J$-colouring. Then, we have

\begin{lemma}\label{Lem-2.5}
\begin{enumerate}\itemsep0mm
\item[(i)] A maximal proper colouring $\varphi:V(G) \mapsto \mathcal{C}$ of a graph $G$ which satisfies a graph theoretical property, say $\mathfrak{P}$, can be minimised to obtain a minimal proper colouring which satisfies $\mathfrak{P}$.
\item[(ii)] A minimal proper colouring $\varphi:V(G) \mapsto \mathcal{C}$ of a graph $G$ which satisfies a graph theoretical property, say $\mathfrak{P}$, can be maximised to obtain a maximal proper colouring which satisfies $\mathfrak{P}$.
\end{enumerate}
\end{lemma}
\begin{proof}
\begin{enumerate}\itemsep0mm
\item[(i)] Consider a maximal proper colouring $\varphi:V(G) \mapsto \mathcal{C}$ of a graph $G$ which satisfies a graph theoretical property say, $\mathfrak{P}$. If a minimum colour set $\mathcal{C}'$, with $|\mathcal{C}'| < |\mathcal{C}|$, such that a minimal proper colouring $\varphi':V(G) \mapsto \mathcal{C}'$ which satisfies the graph theoretical property $\mathfrak{P}$ cannot be found, then $|\mathcal{C}|$ is minimum.
\item[(ii)] Consider a minimal proper colouring $\varphi:V(G) \mapsto \mathcal{C}$ of a graph $G$ which satisfies a graph theoretical property say, $\mathfrak{P}$. If a maximum colour set $\mathcal{C}'$, $|\mathcal{C}'| > |\mathcal{C}|$, such that a maximal proper colouring $\varphi':V(G) \mapsto \mathcal{C}'$ which satisfies the graph theoretical property $\mathfrak{P}$ cannot be found, then $|\mathcal{C}|$ is maximum.
\end{enumerate}
\end{proof}

The following theorem characterises those graphs which admit a $J$-colouring.

\begin{theorem}
A graph $G$ of order $n$ admits a $J$-colouring if and only if $r_\chi(G) = n$.
\end{theorem}
\begin{proof}
If $r_\chi(G)=n$, then every vertex of $G$ belongs to a rainbow neighbourhood. Hence, either the chromatic colouring $\varphi:V(G) \mapsto \mathcal{C}$ is maximal or a maximal colouring $\varphi': V(G)\mapsto \mathcal{C}'$ exists.

An immediate consequence of Definition \ref{Def-1.1} is that if graph $G$ admits a $J$-colouring then each vertex $v\in V(G)$ yields a rainbow neighbourhood. This consequence also follows from the the result that for any connected graph $G$, $\J(G)\le \delta(G)+1$ (see \cite{NKS}). Hence, from Lemma \ref{Lem-2.5} it follows that either the $J$-colouring is minimal or a minimal colouring $\varphi':V(G) \mapsto \mathcal{C}'$ exists such that $r_\chi(G) = n$.
\end{proof}

The following theorem establishes a necessary and sufficient condition for a graph $G$ to have a $J$-colouring with respect to a $\chi^-$-colouring of $G$.

\begin{theorem}\label{Thm-2.6}
A graph $G$ admits a $J$-colouring if and only if each $v \in V(G)$ yields a rainbow neighbourhood with respect to a $\chi^-$-colouring of $G$.
\end{theorem}
\begin{proof}
If  in a $\chi^-$-colouring of $G$, each $v \in V(G)$ yields a rainbow neighbourhood it follows from the second part of Lemma \ref{Lem-2.5} that the corresponding proper colouring can be maximised to obtain a $J$-colouring.

Conversely, assume that a graph $G$ admits a $J$-colouring. Then, it follows from Lemma \ref{Lem-2.5}(i) that the corresponding proper colouring can be minimised to obtain a minimal proper colouring for which each $v\in V(G)$ yields a rainbow neighbourhood. Let the aforesaid set of colours be $\mathcal{C}'$. Assume that a minimum set of colours $\mathcal{C}$ exists which is a $\chi^-$-colouring of $G$ and $|\mathcal{C}|<|\mathcal{C}'|$. It implies that there exists at least one vertex $v \in V(G)$ for which at least one distinct pair of vertices, say $u, w \in N(v)$ exists such that $u$ and $v$ are non-adjacent. Furthermore, $c(u)=c(w)$ under the colouring $\varphi:V(G)\mapsto \mathcal{C}$.

Assume that there is exactly one such $v$ and exactly one such vertex pair $u,w\in N(v)$. But then both $u$ and $w$ yield rainbow neighbourhoods in $G$ under the proper colouring $\varphi:V(G)\mapsto \mathcal{C}$, which is a contradiction to the minimality of $\mathcal{C}'$. By mathematical induction, similar contradictions arise for all vertices similar to $v$. This completes the proof.
\end{proof}

\section{Analysis for Certain Graphs}

Note that we have two types of operations related to graphs, that is: operations on a graph $G$ and operations between two graphs $G$ and $H$. Operations on a graph $G$ result in a well defined derivative of $G$. Examples are the complement graph $G^c$, the line graph $L(G)$, the middle graph $M(G)$, the central graph $C(G)$, the jump graph $J(G)$ and the total graph $T(G)$ and so on. Recall that the jump graph $J(G)$ of a graph $G$ of order $n \geq 3$ is the complement graph of the line graph $L(G)$. Also, note that the line graph is the graphical realisation of edge adjacency in $G$ and the jump graph is the graphical realisation of edge independence in $G$. Some other graph derivative operations are edge deletion, vertex deletion, edge contraction, thorning a graph by pendant vertex addition and so on.

Examples of operations between graphs $G$ and $H$ are, the corona between $G$ and $H$ denoted, $G\circ H$, the join denoted, $G+H$, the disjoint union denoted, $G\cup H$, the cartesian product denoted, $G\square H$ and so on.

\subsection{Operations between certain graphs}

The following result establishes a necessary and sufficient condition for the corona of two graphs $G$ and $H$ to admit a $J$-colouring.

\begin{theorem}\label{Thm-3.1}
If graphs $G$ and $H$ admit $J$-colourings, then $G\circ H$ admits a $J$-colouring if and only if either $G = K_1$ or $\J(G)=\J(H)+1$. 
\end{theorem}
\begin{proof}
\textit{Part 1:} If $G = K_1$ assume $\mathcal{C} = \{c_1,c_2,c_3,\dots,c_{\J(H)}\}$ provides a $J$-colouring of $H$. Colour $K_1$ the colour $c_{\J(H)+1}$. Clearly, $\mathcal{C}'=\mathcal{C}\cup \{c_{\J(H)+1}\}$ is a $J$-colouring of $K_1\circ H$.

\textit{Part 2:} If $G \neq K_1$ and $\J(G) = \J(H) + 1$ let $\mathcal{C} = \{c_1,c_2,c_3,\dots,c_\ell$, $\ell = \J(G)\}$ and $\mathcal{C}' = \{c_1,c_2,c_3,\dots,c_{\ell -1}$, $\ell = \J(G)\}$ provide the $J$-colourings of $G$ and $H$, respectively. Assume that $v \in V(G)$ has $c(v) = c_i$ then colour all $u\in V(H)$ for the copy of $H$ corona'd to $v$ for which $c(u)_{(in~H)} = c_i$, $1 \leq i\leq \ell$, to be $c_{\ell +1}$. Clearly every vertex $v \in V(G)\cup V(H)$ yields a rainbow neighbourhood and $|\mathcal{C}|$ is maximal.

Conversely, let $G\circ H$ admit a $J$-colouring. Then, for any vertex $v\in V(G)$ the subgraph $v\circ H$ holds the condition $c(v)\neq c(u),\ \forall u \in V(H)$. Therefore, either $G=K_1$ or $\J(G)=\J(H)+1$.
\end{proof}

The next corollary requires no proof as it is a direct consequence of Theorem \ref{Thm-3.1}.

\begin{corollary}
If $G\circ H$ admits a $J$-colouring then: $\J(G\circ H) = \J(G)$.
\end{corollary}

The following theorem discusses the admissibility of $J$-colouring by the join of two graphs.

\begin{theorem}\label{Thm-3.2}
If and only if both graphs $G$ and $H$ admit a $J$-colouring, then $G+H$ admits a $J$-colouring.
\end{theorem}
\begin{proof}
Assume that both $G$ and $H$ admit a $J$-colouring. Without loss of generality, let $\J(G)\leq \J(H)$. Assume that $\varphi:V(G)\mapsto \mathcal{C}$, $\mathcal{C} = \{c_1,c_2,c_3,\dots,c_\ell\}$ and $\varphi':V(H)\mapsto \mathcal{C}'$, $\mathcal{C}' = \{c_1,c_2,c_3,\dots,c_{\ell'}\}$ is a $J$-colouring of $G$ and $H$, respectively. For each $v \in V(G)$, $c(v) = c_i$ recolour $c(v)\mapsto c_{i+\ell'}$. Denote the new colour set by $\mathcal{C}_{i+\ell'}$. Clearly, each vertex $v \in V(G)$ is adjacent to at least one of each colour in $G+H$ hence, each such vertex yields a rainbow neighbourhood in $G+H$. Similarly, each vertex $u\in V(H)$ is adjacent to at least one of each colour in $G+H$ and hence each such vertex yields a rainbow neighbourhood in $G+H$. Furthermore, since both $|\mathcal{C}|$, $|\mathcal{C}'|$ is maximal colour sets, the set $|\mathcal{C}_{i+\ell'}\cup \mathcal{C}'|$ is maximal. Therefore, $G+H$ admits a $J$-colouring.

The converse follows trivially from the fact that the additional edges between $G$ and $H$ as defined for join form an edge cut in $G+H$.
\end{proof}

The following result discusses the existence of a $J$-colouring for the Cartesian product of two given graphs.

\begin{theorem}\label{Thm-3.3}
If graphs $G$ and $H$ of order $n$ and $m$ respectively admit a $J$-colouring, then 
\begin{enumerate}\itemsep0mm
\item[(i)] $G\Box H$ admits a $J$-colouring.
\item[(ii)] $\J(G\Box H) = \max\{\J(G), \J(H)\}$
\end{enumerate}
\end{theorem}
\begin{proof}
\begin{enumerate}\itemsep0mm
\item[(i)] Without loss of generality assume $\J(H) \geq \J(G)$. Also, assume that $V(G) = \{v_i: 1\leq i \leq n\}$ and $V(H) = \{u_i:1\leq i \leq m\}$. From the definition of $G\Box H$ it follows that $V(G\Box H) = \{(v_i,u_j): 1\leq i \leq n, 1\leq j \leq m\}$. For $i = 1$, if $u_j\sim u_k$ in $H$, where $\sim$ denotes the adjacency, then $(v_1,u_j)\sim (v_1,u_k)$ and hence we obtain an isomorphic copy of $H$. Such a copy admits a $J$-colouring identical to that of $H$ in respect of the vertex elements $u_1,u_2,u_3,\ldots,u_m$. Now obtain the disjoint union with the copies of $H$ corresponding to $i=2, 3,4,\ldots, n$. Apply the definition of $G\Box H$ for $u_1$ and if $v_i\sim v_j$ in $G$, then $(v_i,u_1)\sim (v_j,u_1)$. An interconnecting copy of $G$ is obtained which result in the first iteration connected graph. Similarly, this copy of $G$ admits a $J$-colouring identical to that of $G$ in respect of the vertex elements $v_1,v_2,v_3,\dots,v_n$. Proceeding iteratively to add all copies of $G$ for $i= 2,3,4,\ldots,n$ in terms of the definition of $G\Box H$, clearly shows that a $J$-colouring is admitted.
\item[(ii)] The second part of the result follows from the similar reasoning used to prove and hence, $\chi(G\Box H)=\max\{\chi(G),\chi(H)\}$.
\end{enumerate}
\end{proof}

\subsection{Operations on certain graphs}

Recall that for any connected graph $G$, $\J(G)\leq \delta(G)+1$ (see \cite{NKS}) and for $n \geq 2,\ \J(P_n) =2$ and $\J^*(P_n)=3$. In view of these results, we have the following results in respect of certain operations on paths and cycles.

\begin{proposition}\label{Prop-3.7}
For a path $P_n$, $n\geq 2$ with edge set consecutively labeled as $e_1,e_2,e_3,\dots,e_{n-1}$ and the corresponding line graph vertices consecutively labeled as $u_1,u_2,u_3,\dots,u_{n-1}$. We have
\begin{enumerate}\itemsep0mm
\item[(i)] $\J(L(P_n))=2$ and $\J^*(L(P_n))=3$.
\item[(ii)] $\J(M(P_2))=2$ and $M(P_n)\ n\geq 3$ does not admit a $J$-colouring and $\J^*(M(P_n))=3$.
\item[(iii)] $\J(T(P_n))=\J^*(T(P_n))=3$.
\item[(iv)] For connectivity, let $n\geq 5$. Then $\J(J(P_5))=3$ and $\J^*(J(P_5))=3$ and for $n \geq 6$, $$\J(J(P_n)) = \J^*(J(P_n)) =
\begin{cases}
\frac{n}{2} & \quad \text{$n$ is even}\\
\lfloor \frac{n}{2}\rfloor & \quad \text{$n$ is odd}.
\end{cases}$$
\item[(v)] $\J(C(P_n)) = \J^*(C(P_n))=3$.
\end{enumerate}
\end{proposition}
\begin{proof}
\begin{enumerate}\itemsep0mm
\item[(i)] Since $L(P_n)=P_{n-1}$, the result follows from the result that for any connected graph $G$, $\J(G)\leq \delta(G)+1$.

\item[(ii)] Since $M(P_2) = P_3$ the result follows from the result that for any connected graph $G$, $\J(G)\leq \delta(G)+1$. For $n\geq 3$, the middle graph contains a triangle hence, $\J(M(P_n)) \geq \chi(M(P_n)) \geq 3$. Also $M(P_n)$ has two pendant vertices therefore $r_\chi(M(P_n)) \neq n$. So $M(P_n)$, $n \geq 3$ does not admit a $J$-colouring. The derivative graph $G'=M(P_n)-\{v_1,v_n\}$ contains a triangle and $\delta(G')=2$. Therefore, $\J^*(M(P_n))=3$.

\item[(iii)] Since $\J(T(P_n)) \leq \delta(\J(T(P_n)) +1=3$ and $T(P_n)$ contains a triangle, $\J(T(P_n))=3$. As $T(P_n)$ has no pendant vertex and contains an odd cycle $C_3$, the result $\J^*(T(P_n))=3$ is immediate.

\item[(iv)] For $P_5$ we have $J(P_5)=P_4$. Hence, the result follows from for any connected graph $G$, $\J(G)\leq \delta(G)+1$. For a path $P_n,\ n\geq 6$ and edge set consecutively labeled as $e_1,e_2,e_3,\dots,e_{n-1}$ and the corresponding line graph vertices consecutively labeled as $u_1,u_2,u_3,\dots,u_{n-1}$, we have the consecutive vertex $\chi^-$-colouring sequence of $J(P_n)$ is given by $c_1,c_1,c_2,c_2,c_3,c_3,\ldots,c_{\frac{n}{2}}$ if $n$ is even and $c_1,c_1,c_2,c_2,c_3,c_3,\ldots,c_{\lfloor \frac{n}{2}\rfloor},c_{\lfloor \frac{n}{2}\rfloor}$ if $n$ is odd. Since the vertices $u_i$, $u_{i+1}$, $1\leq i \leq n-2$ are pairwise not adjacent, the $\chi^-$-colouring is maximal as well. Clearly, every vertex $u_i$ yields a rainbow neighbourhood. Therefore, the result follows.

\item[(v)] Since $C(P_n)$ has no pendant vertex and contains an odd cycle $C_5$, the result is immediate.
\end{enumerate}
\end{proof}

Next, we consider cycles $C_n,\ n\geq 3$. In \cite{NKS}, it is proved that 

\begin{theorem}\label{Thm-3.8}\cite{NKS}
If $C_n$ admits a $J$-colouring then: 
\begin{equation*}
\J(C_n)=
\begin{cases}
3 & \text{if}\ n\equiv 0\,({\rm mod}\ 3)\\
2 & \text{if}\ n\equiv 0\,({\rm mod}\ 2) \text{and}\ n\not\equiv 0\,({\rm mod}\ 3).
\end{cases}
\end{equation*}
\end{theorem}

Analogous to the proof of Theorem 2.7 in [8], we now establish the corresponding results for the derivatives of cycle graphs in the following proposition.

\begin{proposition}\label{Prop-3.9}
For a cycle $C_n, n\geq 3$ and edge set consecutively labeled as $e_1,e_2,e_3,\dots,e_n$ and the corresponding line graph vertices consecutively labeled as $u_1,u_2,u_3,\dots,u_n$, we have
\begin{enumerate}\itemsep0mm 
\item[(i)] $\J(L(C_n)) = \J^*(L(C_n))= 2$ if and only if $n\equiv 0\,({\rm mod}\ 2)$ and $n\not\equiv 0\,({\rm mod}\ 3)$, and $\J(L(C_n))=\J^*(L(C_n))=3$ if and only if $n\equiv 0\,({\rm mod}\ 3)$, else, $L(C_n)$ does not admit a $J$-colouring.
\item[(ii)] For $n \geq 3$, $\J(M(C_n))=\J^*(M(C_n)) = 3$ if $n\equiv 0\,({\rm mod}\ 3)$, or if, $M(C_n)$ for $n\not \equiv 0\,({\rm mod}\ 3)$, and without loss of generality admits the colouring: $c(v_1)=c_1$, $c(u_1)=c_2$, $c(v_2) = c_3$, $c(u_2) = c_1$, $c(v_3) = c_2$, $c(u_3) = c_3,\ldots, c(v_{n-1}) = c_1, c(u_{n-1})= c_2, c(v_n)=c_1, c(u_n) = c_3$, else, $M(C_n)$ does not admit a $J$-colouring.
\item[(iii)]  $\J(T(C_n)) = \J^*(T(C_n)) = 4$  if and only if $n$ is even, else, $T(C_n)$ does not admit a $J$-colouring.
\item[(iv)] For $n \geq 6$, $\J(J(C_n)) = \J^*(J(C_n)) = 
\begin{cases}
\frac{n}{2} & \quad \text{$n$ is even}\\
\lfloor \frac{n}{2}\rfloor & \quad \text{$n$ is odd}.
\end{cases}$.
\item[(v)]  $\J(C(C_n)) = \J^*(C(C_n)) = 3$.
\end{enumerate}
\end{proposition}

\begin{proof}
(i) Because $L(C_n) = C_n$ the result follows from Corollary 3.6. Also because $L(C_n)$ has no pendant edges,  $\J(L(C_n)) = \J^*(L(C_n))$.

(ii) If $M(C_n)$ admits a $J$-colouring then $\J(M(C_n)) \leq \delta(\J(M(C_n)) +1 =3$. For $n\equiv 0\,({\rm mod}\ 3)$, consider the colouring: $c(v_1)=c_1$, $c(u_1)=c_2$, $c(v_2)=c_3$, $c(u_2)=c_1$, $c(v_3)=c_2$, $c(u_3)=c_3,\ldots, c(u_{n-1})= c_1, c(v_n)=c_2, c(u_n) = c_3$.

From the definition of the middle graph, we know that $M(C_n)$ has $n$ triangles stringed so clearly the proper colouring is maximum and all vertices yield a rainbow neighbourhood. Part 2 follows by similar reasoning and hence the result follows. Also, since $M(C_n)$ has no pendant edges,  $\J(M(C_n)) = \J^*(M(C_n))$. In all other cases, $\chi((M(C_n))=4$ and a $J$-colouring does not exist.

(iii) Note that $\J(T(C_n))\leq \delta(\J(T(C_n))+1=5$. Since $T(C_n)$ contains a triangle, $\J(T(C_n)) \geq 3$. Furthermore, $\chi((T(C_n)) = 4$ if and only if $n\equiv 0\,({\rm mod}\ 2)$ and $n\not\equiv 0\,({\rm mod}\ 3)$, and all vertices yield a rainbow neighbourhood. Also, for any set of vertices $V'=\{v_i,v_{i+1},v_{i+2},v_{i+2},v_{i+3},v_{i+4}\}\mapsto  \{v_iv_j: 1 \leq i \leq n,~ 0 \leq j \leq 4,~and~ (i+j)\mapsto (i+j)~(mod~6)\}$, the induced subgraph $\langle V'\rangle \neq K_5$. Therefore, $\J(T(C_n)) = 4$. Also because $T(C_n)$ has no pendant edges,  $\J(T(C_n)) = \J^*(T(C_n))$. Otherwise, $\chi((T(C_n)) = 5$, and not all vertices yield a rainbow neighbourhood and hence a $J$-colouring is not obtained.

(iv) For $n = 5$, $J(C_5) = C_5$ and thus, does not admit a $J$-colouring. For a path $C_n$, $n\geq 6$ and edge set consecutively labeled as $e_1,e_2,e_3,\dots,e_{n-1}$ and the corresponding line graph vertices consecutively labeled as $u_1,u_2,u_3,\dots,u_{n-1}$, we have the consecutive vertex $\chi^-$-colouring sequence of $J(C_n)$ is given by  $c_1,c_1,c_2,c_2,c_3,c_3,\ldots,c_{\frac{n}{2}}$ if $n$ is even and $c_1,c_1,c_2,c_2,c_3,c_3,\ldots,c_{\lfloor \frac{n}{2}\rfloor},c_{\lfloor \frac{n}{2}\rfloor}$ if $n$ is odd ($n-1$ entries). As the vertices $u_i$, $u_{i+1}$, $1\leq i \leq n-2$ are pairwise not adjacent, the $\chi^-$-colouring is maximal as well. Clearly, every vertex $u_i$ yields a rainbow neighbourhood. Therefore, the result follows.

(v) The result is trivial for $C(C_3)$. For $n \geq 4$, $\J(C(C_n)) \leq \delta(\J(C(C_n))+1=3$. Since $\chi((C(C_n)) = 3$ and all vertices yield a rainbow neighbourhood and $C(C_n)$ contains a cycle $C_5$, the result $\J(C(C_n))=3$ holds immediately. Also, since $C(C_n)$ has no pendant edges, $\J(C(C_n)) = \J^*(C(C_n))$.
\end{proof}

\section{Extremal Results for Certain Graphs}

For a graph $G$ of order $n\geq 1$, which admits a $J$-colouring the  minimum (or maximum) number of edges in a subset $E'_k\subseteq E(G)$ whose removal ensures that $\J(G-E'_k)=k,\ 1\leq k \leq \J(G)$, is discussed in this section. These extremal variables are called the minimum (or maximum) rainbow bonding variables and are denoted $r^-_k(G)$ and $r^+_k(G)$, respectively. A graph $G$ which does not admit a $J$-colouring has $r^-_k(G)$ and $r^+_k(G)$ undefined. For such aforesaid graph it is always possible to remove a minimal set of edges, $E''$, which is not necessarily unique such that $G- E''$ admits a $J$-colouring. This is formalised in the next result.

\begin{lemma}\label{Lem-4.1}
For any connected graph $G$ which does not admit a $J$-colouring, a minimal set of edges, $E''$ which is not necessarily unique, can be removed such that $G- E''$ admits a $J$-colouring.
\end{lemma}
\begin{proof}
Since any connected graph $G$ of order $n$ and size $\varepsilon(G)=p$ has a spanning subtree and any tree admits a $J$-colouring, at most $p-(n-1)$ edges must be removed from $G$. Therefore, if $p-(n-1)$ is not a minimal number of edges to be removed then a minimal set of edges $E'$, $|E'| < p - (n-1)$ must exist whose removal results in a spanning subgraph $G'$ which allows a $J$-colouring.
\end{proof}

It is obvious from Lemma \ref{Lem-4.1} that the restriction of connectedness can be relaxed if $G=\bigcup H_i,\ 1\leq i \leq t$ and it is possible that $\J(H_i-E"_i)_{\forall i}=k$, $k$ some integer constant.

It is obvious that for a complete graph $K_n$, $\J(K_n) = n$. To ensure $\J(K_n) = n$, no edges can be removed. Therefore, $r^-_n(K_n) = r^+_n(K_n) = 0$.

\begin{theorem}\label{Thm-4.1}
For a complete graph $K_n$, $n\geq 1$ we have
\begin{enumerate}\itemsep0mm 
\item[(i)] For $n$ is even and $\frac{n}{2}\leq k \leq n$ and $\J(K_n-E'_k)=k$, then $r^-_k(K_n) = n-k$.
\item[(ii)]  For $n$ is odd and $\lceil \frac{n}{2}\rceil \leq k \leq n$, and $\J(K_n - E'_k) = k$, then $r^-_k(K_n) = n-k$.
\item[(iii)]  For $n \in \N$ and $1\leq k \leq n$, and $\J(K_n - E'_k) = k$, then $r^+_k(K_n) = \frac{1}{2}(n+1-k)(n-k)$.
\end{enumerate}
\end{theorem}
\begin{proof}
(i) For $n$ is even and $\frac{n}{2}\leq k \leq n$, exactly $0~or~1~ or~ 2~ or~ 3~or\cdots or~ \frac{n}{2}$ edges between distinct pairs of vertices can be removed to obtain $\J(K_n-E'_k) = n, n-1, n-2, \dots, \frac{n}{2}$. Hence, $r^-_k(K_n) = 0, 1, 2, 3, \dots, \frac{n}{2}$. In other words $r^-_k(K_n) = n - k$, $\frac{n}{2}\leq k \leq n$.

(ii) The result follows through similar reasoning as that in (i).

(iii) In any clique of order $t$, the removal of the $\frac{1}{2}t(t-1)$ edges is the maximum number of edges whose removal renders $\J(\mathfrak{N}_t) =1$ hence, all vertices can be coloured say, $c_1$. Through immediate mathematical induction it follows that we iteratively remove the maximum number of edges $r^+_k(K_n) = 0, 1, 3, 6, 10,\dots, \frac{1}{2}(n+1-k)(n-k)$, $1\leq k\leq n$  of cliques $K_1, K_2, K_3, \dots, K_n$ to obtain $\J(K_n-E'_k) = n, n-1, n-2, \dots,1$. Hence, the result follows.
\end{proof}

\begin{theorem}\label{Thm-4.2}
A graph $G$ of order $n$ which allows a $J$-colouring, has $r^-_k(G) =r^+_k(G)$ if and only if $\J(G)=2$.
\end{theorem}
\begin{proof}
If $\J(G) = 2$ then all edges are incident with colours $c_1, c_2$. Therefore all edges must be removed to obtain the null graph $\mathfrak{N}_0$ for which $\J(\mathfrak{N}_0) =1$. Hence, $r^-_k(G) =r^+_k(G)$.

Conversely, let $r^-_k(G)=r^+_k(G)$. Then, assume that at least one edge say, $e$ is incident with colour $c_3$. It implies that $G$ contains at least a triangle or an odd cycle. Therefore, $\varepsilon(G) \geq 3$. To ensure a proper colouring on removing edge $e$ the colour $c_3$ must change to either $c_1$ or $c_2$ which is always possible. If $\J(G-e) = 2$ then $ r^+_k(G) = 1$ which is a contradiction because any one additional edge may have been removed, implying $ r^+_k(G) \geq 2$. For colours $c_4,c_5,c_6,\dots, \J(G)$, similar contradictions follows through immediate induction. Therefore, if $r^-_k(G) = r^+_k(G)$ then, $\J(G) = 2$.
\end{proof}

\section{Conclusion}

Clearly the cycles for which the the middle graphs  admit a $J$-colouring in accordance with the second part of Proposition \ref{Prop-3.7}(ii) require to be characterised if possible. It follows from Theorem \ref{Thm-4.2} that for the cases $n$ is even and $1 \leq k <\frac{n}{2}$, or $n$ is odd and $1 \leq k < \lceil \frac{n}{2}\rceil$, determining $r^-_k(K_n)$ remains open. It is suggested that an algorithm must be described to obtain these values.

\begin{example}{\rm 
For the complete graph $K_9$ with vertices $v_1,v_2,v_3,\dots, v_9$, Theorem \ref{Thm-4.1}(ii) admits the minimum removal of $r^-_{n,k}(K_n) = 4$ edges to obtain $\J(K_n - E'_k) = 5$. Without loss of generality say the edges were. $v_1v_2$, $v_3v_4$, $v_6v_6$, $v_7v_8$. To obtain $\J(K_n - E'_k) = 4$ we only remove without loss of generality say, the edges $v_7v_9$, $v_8v_9$. To obtain $\J(K_n - E'_k) = 3$ we only remove without loss of generality say, the edges $v_1v_3$, $v_1v_4$, $v_2v_3$, $v_2v_4$. To obtain $\J(K_n - E'_k) = 2$ we only remove without loss of generality say, the edges $v_5v_7$, $v_5v_8$, $v_5v_9$, $v_6v_7$, $v_6v_8$, $v_6v_9$. To obtain $\J(K_n - E'_k) = 1$ we remove all remaining edges. It implies that as $\J(K_n - E'_k)$ iteratively ranges through the values $5, 4, 3, 2, 1$ the value of $r^-_k(K_9)$ ranges through, $4, 6, 10, 16, 36$.}
\end{example}

Determining the range of minimum (maximum) rainbow bonding variables for other classes of graphs is certainly worthy research. For a graph $G$ which does not allow a $J$-colouring it follows from Lemma \ref{Lem-4.1} that a study of $r^-_k(G')$ and $r^+_k(G')$ with $G'$ a maximal spanning subgraph of $G$ which does allow a $J$-colouring, is open.


\begin{thebibliography}{25}

\bibitem {JAG} J A. Gallian, \textbf{A dynamic survey of graph labeling}, Electron. J. combin. DS-6, 2015.

\bibitem{BM1} J. A. Bondy and U. S. R. Murty, \textbf{Graph theory}, Springer, New York, 1976. 

\bibitem{RLB} R. L. Brooks, \textit{On colouring the nodes of a network}, Math. Proc. Cambridge Philos. Soc., \textbf{37}(1941), 194-197.

\bibitem{CZ1} G. Chartrand and P. Zhang, \textbf{Chromatic graph theory}, CRC Press, 2009.

\bibitem{FH} F. Harary, \textbf{Graph theory}, New Age International, New Delhi, 2001.

\bibitem{JT1} T. R. Jensen and B. Toft, {\bf Graph colouring problems}, John Wiley \& Sons, 1995.

\bibitem{KSJ} J. Kok, N. K. Sudev, M. K. Jamil, {\it Rainbow neighbourhoods of graphs}, under review.

\bibitem{NKS} N. K. Sudev, {\it A note on Johan colouring of graphs}, under review.

\bibitem{DBW} D. B. West, {\bf Introduction to graph theory}, Pearson Education Inc., Delhi, 2001.
\end{thebibliography}
\end{document}